\documentclass{amsart}

\usepackage{amsmath}
\usepackage{amssymb}
\usepackage{MnSymbol}
\usepackage{bbm}
\usepackage{epsfig}  		
\usepackage{xypic}
\usepackage{hyperref}
\usepackage{mathrsfs}
\usepackage{lineno}
\usepackage{xcolor}

\newtheorem{mthm}{Theorem}
\renewcommand\themthm{\Alph{mthm}}

\newtheorem{thm}{Theorem}[section]
\newtheorem{prop}[thm]{Proposition}
\newtheorem{lem}[thm]{Lemma}

\theoremstyle{definition}

\newtheorem{example}[thm]{Example}

\theoremstyle{remark}

\newtheorem*{remark}{Remark}

\numberwithin{equation}{section}

\DeclareMathAlphabet{\mathpzc}{OT1}{pzc}{m}{it}
\renewcommand{\mathcal}[1]{\text{$\mathpzc{#1}$}}

\newcommand{\RR}{\mathbbm{R}}

\newcommand{\ZZ}{\mathbbm{Z}}

\newcommand{\ad}{\mathrm{ad}}
\newcommand{\Ad}{\mathrm{Ad}}

\newcommand{\Iso}{\group{Iso}}

\newcommand{\GL}{\group{GL}}

\newcommand{\TT}{\group{T}}

\newcommand{\SO}{\group{SO}}

\newcommand{\OO}{\group{O}}

\newcommand{\Aut}{\group{Aut}}
\newcommand{\Inn}{\group{Inn}}

\newcommand{\Euc}{\group{E}}
\newcommand{\Zen}{\mathrm{Z}}

\newcommand{\group}{\mathrm} 

\newcommand{\ac}[1]{\overline{#1}^{\rm z}}

\renewcommand{\rho}{\varrho}
\renewcommand{\tilde}{\widetilde}

\renewcommand{\epsilon}{\varepsilon}

\newcommand{\zsp}{\mathbf{0}} 
\newcommand{\met}{\langle\cdot,\cdot\rangle}
\newcommand{\one}{\{e\}}

\newcommand{\liealgebra}{\mathcal}

\newcommand{\frg}{\liealgebra{G}}
\newcommand{\frh}{\liealgebra{H}}
\newcommand{\fra}{\liealgebra{A}}
\newcommand{\frb}{\liealgebra{B}}
\newcommand{\frc}{\liealgebra{C}}

\newcommand{\frz}{\liealgebra{Z}}
\renewcommand{\frm}{\liealgebra{M}}

\newcommand{\frk}{\liealgebra{K}}
\newcommand{\frj}{\liealgebra{J}}

\newcommand{\frf}{\liealgebra{F}}

\newcommand{\frs}{\liealgebra{S}}
\newcommand{\frl}{\liealgebra{L}}

\newcommand{\frq}{\liealgebra{Q}}

\newcommand{\frr}{\liealgebra{R}}

\newcommand{\zen}{\liealgebra{Z}}

\newcommand{\euc}{\liealgebra{E}}

\newcommand{\so}{\liealgebra{SO}}

\newcommand{\p}{\mathsf{p}}

\newcommand{\I}{\mathrm{I}}

\renewcommand{\phi}{\varphi}

\newcommand{\gs}{\frg_{\rm s}}

\begin{document}


\title[Homogeneous spaces of finite volume]{Simply connected indefinite homogeneous spaces of finite volume}

\author[Baues]{Oliver Baues}
\address{Oliver Baues, Department of Mathematics, Chemin du Mus\'ee 23, University of Fribourg, CH-1700 Fribourg,  Switzerland} 
\email{oliver.baues@unifr.ch}

\author[Globke]{Wolfgang Globke}
\address{Wolfgang Globke, Faculty of Mathematics, University of Vienna, Oskar-Morgenstern-Platz 1, 1090 Vienna, Austria}
\email{wolfgang.globke@univie.ac.at}

\author[Zeghib]{Abdelghani Zeghib}
\address{Abdelghani Zeghib, \'Ecole Normale Sup\'erieure de Lyon, 
Unit\'e de Math\'ematiques Pures et Appliqu\'ees,
46 All\'ee d'Italie,
69364 Lyon,
France}
\email{abdelghani.zeghib@ens-lyon.fr}

\date{\today}


\subjclass[2010]{Primary 53C50; Secondary 53C30, 57S20}

\begin{abstract}
Let $M$ be a simply connected pseudo-Riemannian homogeneous space of
finite volume with isometry group $G$.
We show that $M$ is compact and that the solvable radical of $G$
is abelian and the Levi factor is a compact semisimple Lie group
acting transitively on $M$.
For metric index less than three, we find that the isometry group of
$M$ is compact itself.
Examples demonstrate that $G$ is not necessarily
compact for higher indices.
To prepare these results, we study Lie algebras with abelian solvable
radical and a nil-invariant symmetric bilinear form. For these,
we derive an orthogonal decomposition into three distinct types of
metric Lie algebras.
\end{abstract}

\maketitle

\tableofcontents


\section{Introduction and main results}
\label{sec:intro}


In this article we are interested in the isometry groups of simply
connected homogeneous pseudo-Riemannian manifolds of finite volume.
D'Ambra \cite[Theorem 1.1]{dambra} showed that a simply connected
compact analytic Lorentzian manifold (not necessarily homogeneous)
has compact isometry group, and she also gave an example of a simply
connected compact analytic manifold of metric signature $(7,2)$ that
has a non-compact isometry group.

Here we study homogeneous spaces for arbitrary metric signature.
Our main tool is the structure theory of the isometry Lie algebras
developed by the authors in \cite{BGZ}. 
The metric on the homogeneous space induces a symmetric bilinear
form on the isometry Lie algebra, and
as shown in \cite{BG,BGZ}, the existence of a finite invariant
measure then implies that this bilinear form is nil-invariant.
The first main result is the following theorem:

\begin{mthm}\label{mthm:geometric}
Let $M$ be a connected and simply connected pseudo-Riemannian
homogeneous space of finite volume, $G=\Iso(M)^\circ$,
and let $H$ be the stabilizer subgroup in $G$ of a point in $M$.
Let $G=KR$ be a Levi decomposition, where $R$ is the solvable radical
of $G$.
Then:
\begin{enumerate}
\item
$M$ is compact.
\item
$K$ is compact and acts transitively on $M$.
\item
$R$ is abelian.
Let $A$ be the maximal compact subgroup of $R$. Then $A=\Zen(G)^\circ$.
More explicitely, $R=A\times V$ where $V\cong\RR^n$ and $V^{K}=\zsp$.
\item
$H$ is connected.
If $\dim R>0$, then $H=(H\cap K) E$, where $E$ and $H\cap K$ are
normal subgroups in $H$, $(H\cap K)\cap E$ is finite,
and $E$ is the graph of a non-trivial homomorphism
$\phi:R\to K$, where the restriction $\phi|_A$ is injective.
\end{enumerate}
\end{mthm}

In Section \ref{sec:simplyconn} we give examples of isometry groups
of compact simply connected homogeneous $M$ with non-compact radical.
However, for metric index $1$ or $2$ the isometry
group of a simply connected $M$ is always compact:

\begin{mthm}\label{mthm:index2}
The isometry group of any simply connected pseudo-Riemannian
homogeneous manifold of finite volume with metric index $\ell\leq2$
is compact.
\end{mthm}

As follows from Theorem \ref{mthm:geometric}, the isometry Lie algebra of a
simply connected pseudo-Riemannian homogeneous space of finite volume
has abelian radical.
This motivates a closer investigation of Lie algebras with abelian
radical that admit nil-invariant symmetric bilinear forms in
Section \ref{sec:abelianradical}.
Our main result is the following algebraic theorem:

\begin{mthm}\label{mthm:abelian_rad1}
Let $\frg$ be a Lie algebra whose solvable radical $\frr$ is
abelian.
Suppose $\frg$ is equipped with a nil-invariant symmetric bilinear
form $\met$ such that the kernel $\frg^\perp$ of $\met$
does not contain a non-trivial ideal of $\frg$.
Let $\frk\times\frs$ be a Levi subalgebra of $\frg$, where $\frk$
is of compact type and $\frs$ has no simple factors of compact
type.
Then $\frg$ is an orthogonal direct product of ideals
\[
\frg = \frg_1 \times \frg_2 \times \frg_3,
\]
with
\[
\frg_1=\frk\ltimes\fra,
\quad
\frg_2=\frs_0,
\quad
\frg_3=\frs_1\ltimes\frs_1^*,
\]
where $\frr=\fra\times\frs_1^*$ and $\frs=\frs_0\times\frs_1$ are
orthogonal direct products,
and $\frg_3$ is a metric cotangent algebra.
The restrictions of $\met$ to $\frg_2$ and $\frg_3$ are invariant
and non-degenerate.
In particular, $\frg^\perp\subseteq\frg_1$.
\end{mthm}

For the definition of metric cotangent algebra, see Section \ref{sec:nilinvariant}.
We call an algebra $\frg_1=\frk\ltimes\fra$ with
$\frk$ semisimple of compact type and $\fra$ abelian a Lie algebra
of \emph{Euclidean type}.
By Theorem \ref{mthm:geometric}, isometry Lie algebras of compact
simply connected pseudo-Riemannian homogeneous spaces are
of Euclidean type.
However, not every Lie algebra of Euclidean type appears as the
isometry Lie algebra of a compact pseudo-Riemannian homogeneous
space.
In fact, this is the case for the Euclidean Lie algebras
$\euc_n=\so_n\ltimes\RR^n$ with $n\neq 3$.

\begin{mthm}\label{thm:noSOnRn2}
The Euclidean group $\Euc_n=\OO_n\ltimes\RR^n$, $n\neq 1,3$,
does not have
compact quotients with a pseudo-Riemannian metric such that
$\Euc_n$ acts isometrically and almost effectively.
\end{mthm}

Note that $\Euc_n$ acts transitively and effectively on compact
manifolds with finite fundamental group, as we remark
at the end of Section \ref{sec:abelianradical}.

\subsection*{Notations and conventions}

For a Lie group $G$, we let $G^\circ$ denote the connected component
of the identity.
For a subgroup $H$ of $G$, we write $\Ad_{\frg}(H)$ for the adjoint
representation of $H$ on the Lie algebra $\frg$ of $G$, to distinguish
it from the adjoint representation $\Ad(H)$ on its own Lie algebra
$\frh$.

The \emph{solvable radical} $R$ of $G$ is the maximal connected
solvable normal subgroup of $G$.
The \emph{solvable radical} $\frr$ of $\frg$ is the maximal solvable
ideal of $\frg$.
The semisimple Lie algebra $\frf=\frg/\frr$ is a direct product
$\frf=\frk\times\frs$, where $\frk$ is a semisimple Lie algebra of
\emph{compact type}, meaning its Killing form is definite, and
$\frs$ is semisimple without factors of compact type.

The center of a group $G$, or a Lie algebra $\frg$, is denoted by
$\Zen(G)$, or $\zen(\frg)$, respectively.
Similarly, the centralizer of a subgroup $H$ in $G$
(or a subalgebra $\frh$ in $\frg$) is denoted by $\Zen_G(H)$
(or $\zen_\frg(\frh)$).

The action of a Lie group $G$ on a homogeneous space $M$ is
\emph{(almost) effective} if the stabilizer of any point in $M$ does
not contain a non-trivial (connected) normal subgroup of $G$.

If $V$ is a $G$-module, then we write
$V^G=\{v\in V\mid gv=v \text{ for all } g\in G\}$ for the
module of \emph{$G$-invariants}.
Similary, $V^\frg=\{v\in V\mid xv=0 \text{ for all } x\in\frg\}$
for a $\frg$-module.

For direct products of Lie algebras $\frg_1$, $\frg_2$ we write
$\frg_1\times\frg_2$, whereas $\frg_1+\frg_2$ or $\frg_1\oplus\frg_2$
refers to sums as vector spaces.

\subsection*{Acknowledgements}

Wolfgang Globke was partially supported by the Australian Research Council grant {DE150101647} and the Austrian Science Foundation FWF grant I 3248.


\section{Nil-invariant bilinear forms}
\label{sec:nilinvariant}

Let $\frg$ be a finite-dimensional real Lie algebra,
let $\Inn(\frg)$ denote the inner auto\-morphism group of $\frg$
and $\ac{\Inn(\frg)}$ its Zariski closure in $\Aut(\frg)$.
A symmetric bilinear form $\met$ on $\frg$ is called \emph{nil-invariant} if
for all $x_1,x_2\in\frg$,
\begin{equation}
\langle \phi x_1, x_2\rangle
=
-\langle x_1,\phi x_2\rangle
\label{eq:nilinvariant}
\end{equation}
for all nilpotent elements $\phi$ of the Lie algebra of
$\ac{\Inn(\frg)}$.
For a subalgebra $\frh$ of $\frg$, we say $\met$ is \emph{$\frh$-invariant}
if  for all $x\in\frh$, $\ad_\frg(x)$ is skew-symmetric for $\met$.

The \emph{kernel} of $\met$ is the subspace
\[
\frg^\perp = \{x\in\frg\mid \langle x,y\rangle=0\text{ for all }y\in\frg\}.
\]
We use a Levi decomposition of $\frg$,
\[
\frg = (\frk\times\frs)\ltimes\frr,
\]
where $\frk$ is semisimple of compact type, $\frs$ is semisimple without
factors of compact type, and $\frr$ is the solvable radical of $\frg$.
Let further $\gs=\frs\ltimes\frr$.

\begin{thm}[\mbox{\cite[Theorem A]{BGZ}}]\label{thm:BGZ_invariance}
Let $\frg$ be a finite-dimensional real Lie algebra with
nil-invariant symmetric bilinear form $\met$.
Let $\met_{\gs}$ denote the restriction of $\met$ to $\gs$.
Then:
\begin{enumerate}
\item
$\met_{\gs}$ is invariant by the adjoint action of $\frg$ on $\gs$.
\item
$\met$ is invariant by the adjoint action of $\gs$.
\end{enumerate}
\end{thm}

This implies some orthogonality relations that will be
useful later on:
\begin{equation}
\frs\perp[\frk,\frg],
\quad
\frk\perp[\frs,\frg].
\label{eq:orthogonal}
\end{equation}


\begin{thm}[\mbox{\cite[Corollary C]{BGZ}}]\label{thm:BGZ_kernel}
Let $\frg$ be a finite-dimensional real Lie algebra with
nil-invariant symmetric bilinear form $\met$, where we further
assume that $\frg^\perp$ does not contain any non-zero ideal of $\frg$.
Let $\zen(\gs)$ denote the center of $\gs$. Then
\[
\frg^\perp\subseteq\frk\ltimes\zen(\gs)
\quad\text{ and }\quad
[\frg^\perp,\gs]\subseteq\zen(\gs)\cap\frg^\perp.
\]
\end{thm}

We say that $\met$ has \emph{relative index} $\ell$ if the induced
scalar product on $\frg/\frg^\perp$ has index $\ell$.
For relative index $\ell\leq 2$, we have a general structure theorem
for $\frg$. 

\begin{thm}[\mbox{\cite[Theorem D]{BGZ}}]\label{thm:BGZ_index2}
Let $\frg$ be a finite-dimensional real Lie algebra with
nil-invariant symmetric bilinear form $\met$ of relative index
$\ell\leq 2$, and assume that $\frg^\perp$ does not contain any non-zero
ideal of $\frg$. Then:
\begin{enumerate}
\item
The Levi decomposition of $\frg$ is a direct sum of ideals
$\frg=\frk\times\frs\times\frr$.
\item
$\frg^\perp$ is contained in $\frk\times\zen(\frr)$ and
$\frg^\perp\cap\frr=\zsp$.
\item
$\frs\perp(\frk\times\frr)$ and $\frk\perp[\frr,\frr]$.
\end{enumerate}
\end{thm}

\subsection{Cotangent algebras}
\label{subsec:cotangent}
Let $\frl$ be a Lie algebra.
A \emph{cotangent algebra} constructed from $\frl$
is a Lie algebra $\frg=\frl\ltimes\frl^*$ where $\frl$ acts on
its dual space $\frl^*$ by its coadjoint representation.
We call $\frg$ a \emph{metric cotangent algebra} if it has a
non-degenerate invariant scalar product $\met$ such that $\frl^*$
is totally isotropic.

%
\subsection[Invariance by $\frg^\perp$]{Invariance by $\boldsymbol{\frg^\perp}$}

We are mainly interested in nil-invariant bi\-linear forms $\met$ on
$\frg$ induced by pseudo-Riemannian metrics on homogeneous spaces.
In this case, $\met$ is invariant by the stabilizer subalgebra
$\frg^\perp$. We can then further sharpen the
statement of Theorem \ref{thm:BGZ_kernel}.

\begin{prop} \label{prop:gperp_inv}
Let $\frg$ and $\met$ be as in Theorem \ref{thm:BGZ_kernel}.
If in addition $\met$ is $\frg^{\perp}$-invariant,
then
\[
[\frg^{\perp}, \gs] = \zsp.
\]
\end{prop}


The proof is based on the following
immediate observations:

\begin{lem}\label{lem0}
Suppose $\met$ is $\frg^{\perp}$-invariant.
Then $ [ [\frk, \frg^{\perp}], \gs]  \subseteq \frg^{\perp} \cap \gs$. 
\end{lem} 

and

\begin{lem} \label{lemA}
Let $\frh$ be any Lie algebra and $V$ a module for $\frh$. Suppose that
the subalgebra $\frq$ of $\frh$ is generated by the subspace $\frm$  of $\frh$. 
Then $\frq \cdot V = \frm \cdot V$. 
\end{lem} 

Together with 

\begin{lem} \label{lemB}
Let $\frk$ be semisimple of compact type and $\frk_{0}$ a subalgebra of $\frk$. 
Then the subalgebra $\frq$ generated $\frm = \frk_{0} + [\frk, \frk_{0}]$ is an ideal of $\frk$. 
\end{lem} 
\begin{proof} 
Put $\frz= \frz_{\frk}(\frk_{0})$. 
Then $[\frz, \frm] \subseteq  \frm$ and  $[ [\frk, \frk_{0}], \frm] \subseteq  \frm +  [\frm, \frm]$. 
Since $\frk =  [\frk, \frk_{0}] + \frz$, this shows $[\frk, \frm] \subseteq \frq$.
Since $\frq$ is linearly spanned by the iterated commutators of elements of $\frm$,
$[\frk, \frq] \subseteq \frq$. 
\end{proof}

\begin{proof}[Proof of Proposition \ref{prop:gperp_inv}] 
Let $\frk_{0}$ be the image of $\frg^{\perp}$ under the projection  homomorphism $\frg \to \frk$. 
Note that by Theorem \ref{thm:BGZ_kernel} above,
$[\frg^{\perp}, \gs] = [\frk_{0}, \gs]$.
Let $\frq \subseteq \frk$ be the  subalgebra generated by
$\frm = \frk_{0} + [\frk, \frk_{0}]$ and consider $V = \gs$ as a module
for $\frq$.
Since $\frq$ is an ideal of  $\frk$, $[\frq, V]$ is a submodule for $\frk$, that is, $[\frk, [\frq, V]] \subseteq  [\frq, V]$.
By Lemmas \ref{lem0}, \ref{lemA} and Theorem \ref{thm:BGZ_kernel}
we have
$[\frq, V] = [\frm, V] \subseteq \frg^{\perp} \cap \zen(\gs)$.
Hence, $\frj =  [\frm, V] \subseteq \frg^{\perp}$ is an ideal in $\frg$, with $\frj \supseteq [\frg^{\perp}, \gs] = [\frk_{0}, \gs]$.
Since $\frg^\perp$ contains no non-trivial ideals of $\frg$ by
assumption, we conclude that $\frj=\zsp$.
\end{proof}

%
%
%
%


\section[Metric Lie algebras with abelian radical]{Metric Lie algebras with abelian radical}
\label{sec:abelianradical}

In this section we study finite-dimensional real Lie algebras
$\frg$ whose solvable radical $\frr$ is abelian and which are
equipped with a nil-invariant symmetric bilinear form $\met$.

\subsection{An algebraic theorem}
The Lie algebras with abelian radical and a nil-invariant
symmetric bilinear form decompose into three distinct types of
metric Lie algebras.

{
\renewcommand{\themthm}{\ref{mthm:abelian_rad1}}
\begin{mthm}
Let $\frg$ be a Lie algebra whose solvable radical $\frr$ is
abelian.
Suppose $\frg$ is equipped with a nil-invariant symmetric bilinear
form $\met$ such that the kernel $\frg^\perp$ of $\met$
does not contain a non-trivial ideal of $\frg$.
Let $\frk\times\frs$ be a Levi subalgebra of $\frg$, where $\frk$
is of compact type and $\frs$ has no simple factors of compact
type.
Then $\frg$ is an orthogonal direct product of ideals
\[
\frg = \frg_1 \times \frg_2 \times \frg_3,
\]
with
\[
\frg_1=\frk\ltimes\fra,
\quad
\frg_2=\frs_0,
\quad
\frg_3=\frs_1\ltimes\frs_1^*,
\]
where $\frr=\fra\times\frs_1^*$ and $\frs=\frs_0\times\frs_1$ are
orthogonal direct products,
and $\frg_3$ is a metric cotangent algebra.
The restrictions of $\met$ to $\frg_2$ and $\frg_3$ are invariant
and non-degenerate.
In particular, $\frg^\perp\subseteq\frg_1$.
\end{mthm}
\addtocounter{mthm}{-1}
}

We split the proof into several lemmas.
Consider the submodules of invariants
$\frr^\frs, \frr^\frk\subseteq\frr$.
Since $\frs$, $\frk$ act reductively, we have
\[
[\frs,\frr]\oplus\frr^\frs=\frr=[\frk,\frr]\oplus\frr^\frk.
\]
Then $\fra=\frr^\frs$, $\frb=[\frs,\frr^\frk]$ and
$\frc=[\frs,\frr]\cap[\frk,\frr]$ are ideals in $\frg$ and
$\frr = \fra\oplus\frb\oplus\frc$.
Recall from Theorem \ref{thm:BGZ_invariance} that $\met$ is
in particular $\frs$- and $\frr$-invariant.

\begin{lem}\label{lem:R=AxB}
$\frc=\zsp$ and $\frr$ is an orthogonal direct sum of ideals in $\frg$
\[
\frr = \fra\oplus\frb
\]
where $[\frk,\frr]\subseteq\fra$ and $[\frs,\frr]=\frb$.
\end{lem}
\begin{proof}
The $\frs$-invariance of $\met$ immediately implies $\fra\perp\frb$.
Since $\frr$ is abelian, $\frr$-invariance implies $\frc\perp\frr$.
Since $\frc\perp(\frs\times\frk)$ by \eqref{eq:orthogonal},
this shows $\frc$ is an ideal contained in
$\frg^\perp$, hence $\frc=\zsp$.
Now $[\frk,\frr]\subseteq\fra$ and $[\frs,\frr]=\frb$ by
definition of $\fra$ and $\frb$.
\end{proof}

\begin{lem}\label{lem:G=KAxSB}
$\frg$ is a direct product of ideals
\[
\frg = (\frk\ltimes\fra)\times(\frs\ltimes\frb),
\]
where $(\frk\ltimes\fra)\perp(\frs\ltimes\frb)$.
\end{lem}
\begin{proof}
The splitting as a direct product of ideals follows from Lemma \ref{lem:R=AxB}.
The orthogonality follows together with \eqref{eq:orthogonal} and
the fact that the $\frs$-invariance of $\met$ implies $\frs\perp\fra$
and $\frk\perp\frb$.
\end{proof}

\begin{lem}\label{lem:SB}
$\frg^\perp\subseteq \frk\ltimes\fra$
and
$\frs\ltimes\frb$ is a non-degenerate ideal of $\frg$.
\end{lem}
\begin{proof}
$\zen(\gs)=\fra$, therefore $\frg^\perp\subseteq\frk\ltimes\fra$
by Theorem \ref{thm:BGZ_kernel}.
Since also $(\frs\ltimes\frb)\perp(\frk\ltimes\fra)$, we have
$(\frs\ltimes\frb)\cap(\frs\ltimes\frb)^\perp\subseteq\frg^\perp\subseteq\frk\ltimes\fra$. It follows that
$(\frs\ltimes\frb)\cap(\frs\ltimes\frb)^\perp=\zsp$.
\end{proof}


To complete the proof of Theorem \ref{mthm:abelian_rad1}, it remains
to understand the structure of the ideal $\frs\ltimes\frb$, which by
Theorem \ref{thm:BGZ_invariance} and the preceding lemmas is a
Lie algebra with an invariant non-degenerate scalar product given by
the restriction of $\met$.

\begin{lem}\label{lem:trivial}
$\frb$ is totally isotropic.
Let $\frs_0$ be the kernel of the $\frs$-action on $\frb$.
Then $\frs_0=\frb^\perp\cap\frs$.
\end{lem}
\begin{proof}
Since $\met$ is $\frr$-invariant and $\frr$ is abelian,
$\frb$ is totally isotropic.
For the second claim, use $\frb\cap\frs^\perp=\zsp$ and the invariance
of $\met$.
\end{proof}

\begin{lem}\label{lem:coadjoint}
$\frs$ is an orthogonal direct product of ideals
$\frs=\frs_0\times\frs_1$
with the following properties:
\begin{enumerate}
\item
$\frs_1\ltimes\frb$ is a metric cotangent algebra.
\item
$[\frs_0,\frb]=\zsp$ and $\frs_0=\frb^\perp\cap\frs$.
\end{enumerate}
\end{lem}
\begin{proof}
The kernel $\frs_0$ of the $\frs$-action on $\frb$ is an ideal in
$\frs$, and by Lemma \ref{lem:trivial}
orthogonal to $\frb$.
Let $\frs_1$ be the ideal in $\frs$ such that $\frs=\frs_0\times\frs_1$.
Then $\frs_0\perp\frs_1$ by invariance of $\met$.

$\frs_1$ acts faithfully on $\frb$ and so $\frs_1\cap\frb^\perp=\zsp$
by Lemma \ref{lem:trivial}.
Moreover, $\frs_1\ltimes\frb$ is non-degenerate since $\frs\ltimes\frb$
is. But $\frb$ is totally isotropic by Lemma \ref{lem:trivial}, so
non-degeneracy implies $\dim\frs_1=\dim\frb$.
Therefore $\frb$ and $\frs_1$ are dually paired by $\met$.

Now identify $\frb$ with $\frs_1^*$ and write $\xi(s)=\langle \xi,s\rangle$
for $\xi\in\frs_1^*$, $s\in\frs_1$.
Then, once more by invariance of $\met$,
\[
[s,\xi](s') = \langle [s,\xi],s'\rangle = \langle \xi,-[s,s']\rangle
= \xi(-\ad(s)s') = (\ad^*(s)\xi)(s')
\]
for all $s,s'\in\frs_1$. So the action of $\frs_1$ on $\frs_1^*$
is the coadjoint action.
This means $\frs\ltimes\frb$ is a metric cotangent algebra
(cf.~Subsection \ref{subsec:cotangent}).
\end{proof}

\begin{proof}[Proof of Theorem \ref{mthm:abelian_rad1}]
The decomposition into the desired orthogonal ideals follows from
Lemmas \ref{lem:G=KAxSB} to \ref{lem:coadjoint}.
The structure of the ideals $\frg_2$ and $\frg_3$ is Lemma
\ref{lem:coadjoint}.
\end{proof}

The algebra $\frg_1$ in Theorem \ref{mthm:abelian_rad1} is of
Euclidean type.
Let $\frg=\frk\ltimes V$, with $V\cong\RR^n$, be an algebra of
Euclidean type and let $\frk_0$ be the kernel of the $\frk$-action
on $V$.
Proposition \ref{prop:gperp_inv} and the fact that the solvable
radical $V$ is abelian immediately give the following:

\begin{prop}\label{prop:Hcompact}
Let $\frg=\frk\ltimes V$ be a Lie algebra of Euclidean type,
and suppose $\frg$ is equipped with a symmetric 
bilinear form that is nil-invariant and $\frg^\perp$-invariant,
such that $\frg^\perp$ does not contain a non-trivial ideal of
$\frg$.
Then
\begin{equation}
\frg^\perp \subseteq \frk_0\times V.
\label{eq:GperpK0V0V1}
\end{equation}
\end{prop}

The following Examples \ref{ex:not_easy} and \ref{ex:not_easy2}
show that in general a metric Lie algebra of Euclidean type cannot
be further decomposed into orthogonal direct sums of metric Lie
algebras.
Both examples will play a role in Section \ref{sec:simplyconn}.

\begin{example}\label{ex:not_easy}
Let $\frk_1=\so_3$, $V_1=\RR^3$, $V_0=\RR^3$
and
$\frg=(\so_3\ltimes V_1)\times V_0$ with the natural action of
$\so_3$ on $V_1$.
Let $\phi:V_1\to V_0$ be an isomorphism and
put
\[
\frh = \{(0,v,\phi(v))\mid v\in V_0\}\subset(\frk_0\ltimes V_1) \times V_0.
\]
We can define a nil-invariant symmetric bilinear form on $\frg$
by identifying $V_1\cong\so_3^*$ and requiring
for $k\in\frk_1$, $v_0\in V_0$, $v_1\in V_1$,
\[
\langle k,v_0+v_1\rangle
=
v_1(k)-\phi^{-1}(v_0)(k),
\]
and further $\frk_1\perp\frk_1$,
$(V_0\oplus V_1)\perp(V_0\oplus V_1)$.
Then $\met$ has signature $(3,3,3)$
and kernel $\frh=\frg^\perp$, which is not an ideal in $\frg$.
Note that $\met$ is not invariant.
Moreover, $\frk_1\ltimes V_1$ is not orthogonal to $V_0$.
A direct factor $\frk_0$ can be added to this example at liberty.
\end{example}

\begin{example}\label{ex:not_easy2}
We can modify the Lie algebra $\frg$ from Example \ref{ex:not_easy} 
by embedding
the direct summand $V_0\cong\RR^3$ in a torus subalgebra in a semisimple Lie
algebra $\frk_0$ of compact type, say $\frk_0=\so_6$, to obtain a
Lie algebra $\frf=(\frk_1\ltimes V_1)\times\frk_0$.
We obtain a nil-invariant symmetric bilinear form of signature
$(15,3,3)$ on $\frf$ by extending $\met$ by a definite form on a
vector space complement of $V_0$ in $\frk_0$. The kernel of the
new form is still $\frg^\perp=\frh$.
\end{example}

\subsection{Nil-invariant bilinear forms on Euclidean algebras}

A \emph{Euclidean algebra} is a Lie algebra
$\euc_n=\so_n\ltimes\RR^n$,
where $\so_n$ acts on $\RR^n$ by the natural action.

By a \emph{skew pairing} of a Lie algebra $\frl$ and an $\frl$-module
$V$ we mean a bilinear map $\met:\frl\times V \to \RR$ such that
$\langle x,  y v \rangle = -\langle y, x v\rangle$
for all $x,y\in\frl$, $v\in V$.
Note that any nil-invariant symmetric bilinear form on
$\frg=\frk\ltimes\RR^n$ yields a skew pairing of $\frk$ and $\RR^n$.

\begin{prop}[\mbox{\cite[Proposition A.5]{BGZ}}]\label{prop:so3_skew_module}
Let $\met: \so_3 \times V \to \RR$ be a skew pairing for the (non-trivial) irreducible module $V$.
If the skew pairing is non-zero, then $V$ is isomorphic to the adjoint representation
of  $\so_3$ and $\met$ is proportional to the Killing form. 
\end{prop}

\begin{example}\label{ex:so3R3}
Consider $\frg=\so_3\ltimes \RR^n$ with a nil-invariant symmetric bilinear
form $\met$, and assume that the action of $\so_3$ is irreducible.
By Proposition \ref{prop:so3_skew_module}, either
$\so_3\perp\RR^n$, 
or $n=3$ and $\so_3$ acts by its coadjoint representation on
$\RR^3\cong\so_3^*$, and $\met$ is the dual pairing.
In the first case, $\RR^n$ is an ideal in $\frg^\perp$, and in the
second case, $\met$ is invariant and non-degenerate.
\end{example}

\begin{example}\label{ex:so4R4}
Let $\frg$ be the Euclidean Lie algebra $\so_4\ltimes \RR^4$ with a
nil-invariant symmetric bilinear form $\met$.
Since $\so_4\cong\so_3\times\so_3$, and here both factors
$\so_3$ act irreducibly on $\RR^4$, it follows from
Example \ref{ex:so3R3} that in $\frg$, $\RR^4$ is orthogonal to both
factors $\so_3$, hence to all of $\so_4$.
In particular, $\RR^4$ is an ideal contained in $\frg^\perp$.
\end{example}

\begin{thm}\label{thm:noSOnRn}
Let $\met$ be a nil-invariant symmetric bilinear form on the
Euclidean Lie algebra $\so_n\ltimes\RR^n$ for $n\geq 4$.
Then the ideal $\RR^n$ is contained in $\frg^\perp$.
\end{thm}
\begin{proof}
For $n=4$, this is Example \ref{ex:so4R4}. So assume $n>4$.
Consider an embedding of $\so_4$ in $\so_n$ such that
$\RR^n=\RR^4\oplus\RR^{n-4}$, where $\so_4$ acts on $\RR^4$ in the
standard way and trivially on $\RR^{n-4}$.
By Example \ref{ex:so4R4}, $\so_4\perp\RR^4$.
Since $\RR^{n-4}\subseteq[\so_n,\RR^n]$, the nil-invariance of
$\met$ implies $\so_4\perp\RR^{n-4}$.
Hence $\RR^n\perp\so_4$.

The same reasoning shows that $\Ad(g)\so_4\perp\RR^n$, where
$g\in\SO_n$.
Then $\frb = \sum_{g\in\SO_n}\Ad(g)\so_4$ is orthogonal to $\RR^n$.
But $\frb$ is clearly an ideal in
$\so_n$, so $\frb=\so_n$ by simplicity of $\so_n$ for $n>4$.
\end{proof}




{
\renewcommand{\themthm}{\ref{thm:noSOnRn2}}
\begin{mthm}
The Euclidean group $\Euc_n=\OO_n\ltimes\RR^n$, $n\neq 1, 3$,
does not have
compact quotients with a pseudo-Riemannian metric such that
$\Euc_n$ acts isometrically and almost effectively.
\end{mthm}
\addtocounter{mthm}{-1}
}
\begin{proof}
For $n>3$, such an action of $\Euc_n$ would induce
a nil-invariant symmetric bilinear form on the Lie algebra
$\so_n\ltimes\RR^n$ without non-trivial ideals in its kernel,
contradicting Theorem \ref{thm:noSOnRn}.

For $n=2$, the Lie algebra $\euc_2$ is solvable, and hence by
Baues and Globke \cite{BG}, any nil-invariant symmetric bilinear
form must be invariant. For such a form, the ideal $\RR^2$ of $\euc_2$
must be contained in $\euc_2^\perp$, and therefore action cannot be
effective.

Note that $\euc_3$ is an exception, as it is the metric cotangent
algebra of $\so_3$.
\end{proof}

\begin{remark}
Clearly the Lie group $\Euc_n$ admits compact quotient manifolds on
which $\Euc_n$ acts almost effectively.
For example take the quotient by a subgroup $F\ltimes\ZZ^n$, where
$F\subset\OO_n$ is a finite subgroup preserving $\ZZ^n$.
Compact quotients with finite fundamental group also exist.
For example, 
for any non-trivial homomorphism $\phi:\RR^n\to\OO_n$, the graph
$H$ of $\phi$ is a closed subgroup of $\Euc_n$ isomorphic
to $\RR^n$, and the quotient $M=\Euc_n/H$ is compact (and diffeomorphic
to $\OO_n$).
Since $H$ contains no non-trivial normal subgroup of $\Euc_n$,
the $\Euc_n$-action on $M$ is effective.
Theorem \ref{thm:noSOnRn2} tells us that none of these quotients
admit $\Euc_n$-invariant pseudo-Riemannian metrics.
\end{remark}


\section[Simply connected spaces]{Simply connected compact homogeneous spaces with indefinite metric}
\label{sec:simplyconn}

Let $M$ be a connected and simply connected pseudo-Riemannian
homogeneous space of finite volume.
Then we can write
\begin{equation}
M=G/H
\label{eq:MGH}
\end{equation}
for a connected Lie group $G$ acting almost
effectively and by isometries on $M$, and $H$ is a closed subgroup
of $G$ that contains no non-trivial connected normal subgroup of $G$
(for example, $G=\Iso(M)^\circ$).
Note that $H$ is connected since $M$ is simply connected.

Let $\frg$, $\frh$ denote the Lie algebras of $G$, $H$, respectively.
Recall that the pseudo-Riemannian metric on $M$ induces an $\frh$-invariant
and nil-invariant symmetric bilinear form $\met$ on $\frg$, and the
kernel of $\met$ is precisely $\frg^\perp=\frh$ and contains no
non-trivial ideal of $\frg$.

We decompose $G=KSR$, where $K$ is a compact
semi\-simple subgroup, $S$ is a semisimple subgroup without compact factors, $R$ the solvable radical of $G$

\begin{prop}\label{prop:S_trivial_M_compact}
The subgroup $S$ is trivial and $M$ is compact. 
\end{prop}
\begin{proof}
As $M$ is simply connected, $H=H^\circ$.
Now $H\subseteq K R$ by Theorem \ref{thm:BGZ_kernel}.
On the other hand, since $M$ has finite invariant volume, the Zariski
closure of $\Ad_\frg(H)$ contains $\Ad_\frg(S)$, see
Mostow \cite[Lemma 3.1]{mostow3}.
Therefore $S$ must be trivial.
It follows from Mostow's result \cite[Theorem 6.2]{mostow4} on quotients of
solvable Lie groups that $M=(KR)/H$ is compact.
\end{proof}

We can therefore restrict ourselves in \eqref{eq:MGH} to groups
$G = K R$
and connected uniform subgroups $H$ of $G$.

The structure of a general compact homogeneous manifold with finite
fundamental group is surveyed in Onishchik and Vinberg \cite[II.5.\S 2]{OV1}.
In our context it follows that
\begin{equation}
G = L\ltimes V
\label{eq:OV}
\end{equation}
where $V$ is a normal subgroup isomorphic to $\RR^n$ and
$L=KA$ is a maximal compact subgroup of $G$.
The solvable radical is $R=A\ltimes V$. Moreover, $V^L=\zsp$.
By a theorem of Montgomery \cite{montgomery} (also \cite[p.~137]{OV1}), $K$ acts transitively on $M$.
\\

The existence of a $G$-invariant metric on $M$ further restricts the
structure of $G$.

\begin{prop}\label{prop:rad_abelian}
The solvable radical $R$ of $G$ is abelian.
In particular, $R=A\times V$, $V^K=\zsp$
and $A=\Zen(G)^\circ$.
\end{prop}
\begin{proof}
Let $\Zen(R)$ denote the center of $R$ and $N$ its nilradical.
Since $H$ is connected, $H\subseteq K\Zen(R)^\circ$ by Theorem
\ref{thm:BGZ_kernel}.
Hence there is a surjection
$G/H\to G/(K\Zen(R)^\circ)=R/\Zen(R)^\circ$.
It follows that $\Zen(R)^\circ$ is a 
connected uniform subgroup. Therefore the nilradical $N$ of $R$ 
is $N=T\Zen(R)^\circ$ for some compact torus $T$.
But a compact subgroup of $N$ must be central in $R$,
so $T\subseteq\Zen(R)$.
Hence $N\subseteq\Zen(R)$, which means $R=N$ is abelian.
\end{proof}

Combined with \eqref{eq:OV}, we obtain
\begin{equation}
G = K R = (K_0 A)\times(K_1\ltimes V),
\label{eq:GKR}
\end{equation}
with $K=K_0\times K_1$, $R=A\times V$, where
$K_0$ is the kernel of the $K$-action on $V$.\\

%


For any subgroup $Q$ of $G$ we write $H_{Q}=H\cap Q$.

\begin{lem}\label{lem:HKHLnormal}
$[H,H]\subseteq H_K$.
In particular, $H_K$ is a normal subgroup of $H$.
\end{lem}
\begin{proof}
By Proposition \ref{prop:Hcompact} and the connectedness of $H$,
we have $H\subseteq K_0 R$.
Since $R$ is abelian, $[H,H]\subseteq H_{K_0}$.
\end{proof}

If $G$ is simply connected, we have $K\cap R=\one$. Then let
$\p_K$, $\p_R$ denote the projection maps from $G$ to $K$, $R$.

\begin{lem}\label{lem:surjective_R}
Suppose $G$ is simply connected.
Then $\p_R(H)=R$.
\end{lem}
\begin{proof}
Since $K$ acts transitively on $M$, we have $G=KH$.
Then $R=\p_R(G)=\p_R(H)$.
\end{proof}

\begin{prop}\label{prop:HKE}
Suppose $G$ is simply connected.
Then the stabilizer $H$ is a semidirect product $H=H_K\times E$, where $E$ is the graph of a homomorphism $\phi:R\to K$ that is non-trivial if $\dim R>0$.
Moreover, $\phi(R\cap H)=\one$.
\end{prop}
\begin{proof}
The subgroup $H_{K}$ is a maximal compact subgroup of the stabilizer
$H$. By Lemma \ref{lem:HKHLnormal},
$H=H_K\times E$ for some normal subgroup $E$
diffeomorphic to a vector space.
By Lemma \ref{lem:surjective_R}, $H$ projects onto $R$ with kernel
$H_K$, so that $E\cong R$.
Then $E$ is necessarily the graph of a homomorphism $\phi:R\to K$.
If $\dim R>0$, then $\phi$ is non-trivial, for otherwise
$R\subseteq H$, in contradiction to the almost effectivity of the
action.
Observe that $R\cap H\subseteq E$.
Therefore $\phi(R\cap H)\subseteq H_K\cap E=\one$.
\end{proof}

Now we can state our main result:

{\renewcommand{\themthm}{\ref{mthm:geometric}}
\begin{mthm}
Let $M$ be a connected and simply connected pseudo-Riemannian
homogeneous space of finite volume, $G=\Iso(M)^\circ$,
and let $H$ be the stabilizer subgroup in $G$ of a point in $M$.
Let $G=KR$ be a Levi decomposition, where $R$ is the solvable radical
of $G$.
Then:
\begin{enumerate}
\item
$M$ is compact.
\item
$K$ is compact and acts transitively on $M$.
\item
$R$ is abelian.
Let $A$ be the maximal compact subgroup of $R$. Then $A=\Zen(G)^\circ$.
More explicitely, $R=A\times V$ where $V\cong\RR^n$ and $V^{K}=\zsp$.
\item
$H$ is connected.
If $\dim R>0$, then $H=(H\cap K) E$, where $E$ and $H\cap K$ are
normal subgroups in $H$, $(H\cap K)\cap E$ is finite,
and $E$ is the graph of a non-trivial homomorphism
$\phi:R\to K$, where the restriction $\phi|_A$ is injective.
\end{enumerate}
\end{mthm}
\addtocounter{mthm}{-1}
}
\begin{proof}
Claims (1), (2) and (3) follow from
Proposition \ref{prop:S_trivial_M_compact},
Proposition \ref{prop:rad_abelian}
and \eqref{eq:OV},
applied to $G=\Iso(M)^\circ$.

For claim (4), let $\tilde{G}$ be the universal cover of $G$.
Since $G$ acts effectively on $M$, $\tilde{G}$ acts almost effectively
on $M$ with stabilizer $\tilde{H}$, the preimage of $H$ in $\tilde{G}$.
Let $\tilde{\phi}:\tilde{R}\to\tilde{K}$ be the homomorphism given by
Proposition \ref{prop:HKE} for $\tilde{G}$.
Then $\tilde{R}=\tilde{A}\oplus V$,
with $\tilde{A}\cong\RR^k$ for some $k$, and $R=\tilde{R}/Z$ for some
central discrete subgroup $Z\subset\tilde{A}\cap\tilde{H}$.
Since $Z\subset\tilde{R}\cap\tilde{H}$ we have
$Z\subseteq\ker\tilde{\phi}$.
Therefore $\tilde{\phi}$ descends to a homomorphism
$R\to\tilde{K}$, and by composing with the canonical projection
$\tilde{K}\to K$,
we obtain a homomorphism $\phi:R\to K$ with the desired properties.
Observe that $\ker\phi|_A\subset A\cap H$ is a normal subgroup in $G$.
Hence it is trivial, as $G$ acts effectively.
\end{proof}

Now assume further
that the index of the metric on $M$ is $\ell\leq 2$.
Theorem \ref{thm:BGZ_index2} has strong consequences in the simply
connected case.

{\renewcommand{\themthm}{\ref{mthm:index2}}
\begin{mthm}
The isometry group of any simply connected pseudo-Riemannian
homogeneous manifold of finite volume and metric index $\ell\leq2$
is compact.
\end{mthm}
}
\begin{proof}
We know from Theorem \ref{mthm:geometric} that $M$ is compact.
Let $G=\Iso(M)^\circ$, with $G=KR$ as before.
By Theorem \ref{thm:BGZ_index2}, $R$ commutes with $K$
and thus $R=A$ by part 3 of Theorem \ref{mthm:geometric}.
It follows that $G=K A$ is compact.

Then $K$ is a characteristic subgroup of $G$ which also acts
transitively on $M$.
Therefore we may identify $\mathrm{T}_x M$ at $x\in M$ with
$\frk/(\frh\cap\frk)$, where $\frk$ is the Lie algebra of $K$.
Hence the isotropy representation of the stabilizer $\Iso(M)_x$
factorizes over a closed subgroup of the automorphism group of $\frk$.
As this latter group is compact, the isotropy representation
has compact closure in $\GL(\mathrm{T}_x M)$.
If follows that there exists a
Riemannian metric on $M$ that is preserved by $\Iso(M)$.
Hence $\Iso(M)$ is compact.
\end{proof}

\begin{remark}
Note that in fact the isometry group of every compact analytic simply
connected pseudo-Riemannian manifold has finitely
many connected components (Gromov \cite[Theorem 3.5.C]{gromov}).
\end{remark}


For indices higher than two,
the identity component of the iso\-metry group of a
simply connected $M$ can be non-compact.
This is demonstrated by the examples below in which we construct some
$M$ on which a non-compact group acts isometrically.
The following Lemma \ref{lem:no_compact} then ensures that these
groups cannot be contained in any compact Lie group.

\begin{lem}\label{lem:no_compact}
Assume that the action of $K$ on $V$ in the semidirect product
$G=K\ltimes V$ is non-trivial.
Then $G$ cannot be immersed in a compact Lie group.
\end{lem}
\begin{proof}
Suppose there is a compact Lie group $C$ that contains $G$ as a
subgroup.
As the action of $K$ on $V$ is non-trivial, there exists an
element $v\in V\subseteq C$ such that $\Ad_{\frc}(v)$ has non-trivial
unipotent Jordan part. But by compactness of $C$, every $\Ad_{\frc}(g)$,
$g\in C$, is semisimple, a contradiction.
\end{proof}

\begin{example}\label{ex:not_easy_group}
Start with
$G_1=(\tilde{\SO}_3\ltimes\RR^3)\times\TT^3$,
where $\tilde{\SO}_3$ acts on $\RR^3$ by the coadjoint action,
and let $\phi:\RR^3\to\TT^3$
be a homomorphism with discrete kernel.
Put
\[
H = \{(\I_3,v,\phi(v))\mid v\in\RR^3\}.
\]
The Lie algebras $\frg_1$ of $G_1$ and $\frh$ of $H$ are the
corresponding Lie algebras from Example \ref{ex:not_easy}.
We can extend the nil-invariant scalar product $\met$ on $\frg_1$
from Example \ref{ex:not_easy} to a left-invariant tensor on $G_1$,
and thus obtain a $G_1$-invariant pseudo-Riemannian metric of signature
$(3,3)$ on the quotient $M_1=G_1/H=\tilde{\SO}_3\times\TT^3$.
Here, $M_1$ is a non-simply connected manifold with a non-compact
connected stabilizer.

In order to obtain a simply connected space, embed $\TT^3$ in a simply
connected compact semisimple group $K_0$, for example
$K_0=\tilde{\SO}_6$, so that $G_1$ is embedded in
$G=(\tilde{\SO}_3 \ltimes \RR^3) \times K_0$.
As in Example \ref{ex:not_easy2}, we can extend $\met$ from $\frg_1$
to $\frg$, and thus obtain a compact simply connected pseudo-Riemannian
homogeneous manifold $M=G/H=\tilde{\SO}_3\times K_0$.
%
%
\end{example}

\begin{example}
Example \ref{ex:not_easy_group} can be generalized by replacing
$\tilde{\SO}_3$ by any simply connected compact semi\-simple group
$K$, acting by the coadjoint representation on $\RR^d$, where
$d=\dim K$, and trivially on $\TT^d$.
Define $H$ similarly as a graph in $\RR^d\times\TT^d$,
and embed $\TT^d$ in a simply connected compact semisimple Lie group
$K_0$.
\end{example}




\end{document}